\documentclass{amsart}

\usepackage{amsmath,amssymb,amsfonts,amsthm,amscd,amstext,amsxtra,amsopn,array,url,verbatim,mathrsfs,enumerate,anysize,enumitem,graphicx,xfrac,mathtools}
\usepackage[table,xcdraw]{xcolor}
\usepackage{colortbl}

\newtheorem{theorem}{Theorem}[section]
\newtheorem{lemma}[theorem]{Lemma}
\newtheorem{corollary}[theorem]{Corollary}
\newtheorem{proposition}[theorem]{Proposition}

\theoremstyle{definition}

\theoremstyle{remark}
\newtheorem{remark}[theorem]{Remark}

\numberwithin{equation}{section}

\makeatletter
\def\imod#1{\allowbreak\mkern3mu({\operator@font mod}\,\,#1)}
\makeatother

\newcommand{\bigO}{\mathcal{O}}

\newcommand{\ub}[1]{\underline{#1}}

\DeclarePairedDelimiter{\floor}{\lfloor}{\rfloor}
\DeclarePairedDelimiter{\fp}{\{}{\}}

\newcommand{\N}{\mathbb{N}} 

\begin{document}

\title[Explicit Versions of Burgess' Bound]{An Investigation Into Explicit Versions of Burgess' Bound}



\author[F. J. Francis]{Forrest J. Francis}
\address{School of Science \\
        UNSW Canberra at ADFA \\
       Northcott Drive, Campbell, ACT 2600 \\
        Australia}
\email{f.francis@student.adfa.edu.au}

\subjclass[2010]{Primary 11L40, 11Y60}
\keywords{Dirichlet Character, Burgess' bounds, non-residues, norm-Euclidean cyclic fields}

\date{}

\dedicatory{}

\begin{abstract}
    Let $\chi$ be a Dirichlet character modulo $p$, a prime. In applications, one often needs estimates for short sums involving $\chi$. One such estimate is the family of bounds known as \emph{Burgess' bound}. In this paper, we explore several minor adjustments one can make to the work of Enrique Trevi\~no \cite{trevino2015} on explicit versions of Burgess' bound. For an application, we investigate the problem of the existence of a $k$th power non-residue modulo $p$ which is less than $p^\alpha$ for several fixed $\alpha$. We also provide a quick improvement to the conductor bounds for norm-Euclidean cyclic fields found in \cite{lezowski2017}.
\end{abstract}
\maketitle

\section{Introduction}

Let $\chi$ be a Dirichlet character modulo $q$. 
It is often useful to know the size of \emph{short} character sums, i.e., sums of the form
\begin{equation}\label{scs}
   S_\chi(M,N) := \sum_{n=M+1}^{M+N} \chi(n)
\end{equation}
where $M$, $N$ are real numbers with $N < q$. A trivial bound for \eqref{scs} is simply $N$, since a Dirichlet character takes values which are either roots of unity or 0. On the other hand, we may estimate \eqref{scs} entirely in terms of $q$; the standard estimate in this direction is the P\'olya--Vinogradov inequality. The following explicit version of this inequality is due to Frolenkov and Soundararajan \cite{frolenkov2013}. 

\begin{theorem}{\textnormal{\cite[Theorem 2]{frolenkov2013}}}\label{explicitpv}
For a primitive Dirichlet character modulo $q>q_0$, we have
\[\left| S_\chi(M,N) \right| \leq \alpha_1\sqrt{q}\log{q} + \sqrt{q},\]
where
\begin{equation*}
(\alpha_1, q_0) = \begin{cases} 
\left(\frac{2}{\pi^2}, 1200\right) 
    & \text{if } \chi(-1) = 1 \\
&\\  
\left(\frac{1}{2\pi}, 40\right)
    & \text{if } \chi(-1)= -1.
\end{cases}
\end{equation*}

\end{theorem}
In this work, we are concerned with characters of large modulus. Following the proof of Theorem \ref{explicitpv} provided in \cite[pg.\! 278]{frolenkov2013} closely, we may take a constant smaller than 1 in front of $\sqrt{q}$ whenever $q$ is bounded below. For an adjusted version of Theorem \ref{explicitpv} that does not bound $q$ below, see \cite[Lemma 3]{lapkova2018}. 

\begin{corollary}\label{pvlargeq}
Let $\chi$ be a primitive Dirichlet character modulo $q > q_0$, and $\alpha_1$ be the constant in Theorem \ref{explicitpv}. Then,

\begin{equation*}
    \left| S_\chi(M,N) \right| \leq \alpha_1\sqrt{q}\log{q} +  \alpha_2\sqrt{q},
\end{equation*}
where 
\[\alpha_2 := \alpha_1\cdot\begin{cases} \left(\log\left(\frac{\pi^4}{16} + \frac{5.075\pi^2}{\sqrt{q_0}} +\frac{103.0225}{q_0}\right) + 2.8650\right)  & \text{if } \chi(-1) = 1 \\  
&\\  
\left(\log\left(\pi^2 + \frac{20.30\pi}{\sqrt{q_0}} +\frac{103.0225}{q_0}\right) + 2.8650\right)  & \text{if } \chi(-1) = 1.
\end{cases}\]

\end{corollary}
For either parity, $\alpha_2 \leq 1$ for $q_0 \geq 854$. However, these savings are slight, even for very large $q_0$, since the limiting value of $\alpha_2$ is greater than $0.9466$ (for even $\chi$) and $0.8203$ (for odd $\chi$). 

Between the trivial estimate, which is entirely in terms of $N$, and the P\'olya--Vinogradov inequality, which is entirely in terms of $q$, we have a family of hybrid estimates due to D. A. Burgess (see, e.g. \cite{burgess1962}, \cite{burgess1963}, \cite{burgess1986}) which take the following form if $q = p$, a prime. 

\begin{theorem}{\textnormal{\cite[Theorem 1]{burgess1963}}}
Let $\chi$ be a non-principal character modulo $p$. Then, if $r$ is a positive integer,  
\[S_\chi(M,N) \ll N^{1-\frac{1}{r}}p^\frac{r+1}{4r^2}\log{p},\]
for any non-negative integers $M,N$.  
\end{theorem}

Proving Burgess' bound requires the power of an estimate derived from the Weil bound \cite{weil1948}. For our purposes, we use an explicit variant of this estimate established by Trevi\~no \cite{trevino2015}. 

\begin{theorem}{\textnormal{\cite[Theorem 1.2]{trevino2015}}}\label{trevinoweil}
Let $p$ be a prime, $r$ a positive real number, and $B$ a positive real number satisfying $r \leq 9B$. Let $\chi$ be a non-principal character modulo $p$. Then 
\[\sum_{x\imod{p}} \left|\sum_{1\leq b \leq B} \chi(x+b)\right|^{2r} \leq (2r-1)!!B^r p + (2r-1)B^{2r} \sqrt{p},\]
where $(2r-1)!! = (2r-1)(2r-3)\dots(1)$. 
\end{theorem}

This estimate is used in \cite{trevino2015}  to establish an explicit version of Burgess' bound for prime moduli. In particular, the following is determined.

\begin{theorem}{\textnormal{\cite[Theorem 1.4]{trevino2015}}}\label{1overr}
Let $p$ be a prime and $2 \leq r \leq 10$ be an integer. Let $\chi$ be a non-principal character modulo $p$. Let $M$ and $N$ be non-negative integers. Let $p_0$ be a positive real number. Then, for $p \geq p_0$, we can determine a constant $c(r)$ depending on $p_0$ and $r$ such that
\begin{equation*}
\left| S_\chi(M,N) \right| < c(r) N^{1-\frac{1}{r}}p^\frac{r+1}{4r^2}\left(\log{p}\right)^\frac{1}{r}.
\end{equation*}
\end{theorem}
The exponent on the $\log{p}$ can be improved by placing a mild condition on the size of $N$. 
\begin{theorem}{\textnormal{\cite[Theorem 1.6]{trevino2015}}}\label{1over2r}
Let $p$ be a prime and $2 \leq r \leq 10$ be an integer. Let $\chi$ be a non-principal character modulo $p$. Let $M$ and $N$ be non-negative integers with $1 \leq N \leq 2p^{\frac{1}{2}+\frac{1}{4r}}$. Let $p_0$ be a positive real number. Then, for $p \geq p_0$, we can determine a constant $C(r)$ depending on $p_0$ and $r$ such that
\begin{equation*}
\left| S_\chi(M,N) \right| < C(r) N^{1-\frac{1}{r}}p^\frac{r+1}{4r^2}\left(\log{p}\right)^\frac{1}{2r}.
\end{equation*}
\end{theorem}

By leveraging the $r=2$ case of Theorem \ref{1overr} against Theorem \ref{1over2r}, Trevi\~no notes that the restriction on $N$ in Theorem \ref{1over2r} can be omitted \cite{trevino2015}. 

\begin{corollary}\label{targetresult}
Let $p$ be a prime and $3 \leq r \leq 10$ be an integer. Let $\chi$ be a non-principal character modulo $p$. Let $M$ and $N$ be non-negative integers with $ 1 \leq N \leq 2p^{\frac{1}{2}+\frac{1}{4r}}$. Let $p_0 \geq 10^{10}$ be a positive real number. Then, for $p \geq p_0$, the constant $C(r)$ in Theorem \ref{1over2r} is such that
\begin{equation*}
\left| S_\chi(M,N) \right| < C(r) N^{1-\frac{1}{r}}p^\frac{r+1}{4r^2}\left(\log{p}\right)^\frac{1}{2r}.
\end{equation*}
\end{corollary}
The constants $c(r)$ and $C(r)$ as provided by \cite{trevino2015} are reproduced in Tables \ref{trevinoconstantsr} and \ref{trevinoconstants2r}.

\begin{table}[ht]
\centering
\caption{Values for the constants $c(r)$ provided by Trevi\~no.}
\label{trevinoconstantsr}
\begin{tabular}{rrrr}
\hline
\rowcolor[HTML]{EFEFEF} 
$r$  & $p_0 = 10^{7}$ & $p_0 = 10^{10}$ & $p_0 = 10^{20}$ \\ \hline
2  & 2.7381 & 2.5173 & 2.3549 \\
3  & 2.0197 & 1.7385 & 1.3695 \\
4  & 1.7308 & 1.5151 & 1.3104 \\
5  & 1.6107 & 1.4572 & 1.2987 \\
6  & 1.5482 & 1.4274 & 1.2901 \\
7  & 1.5052 & 1.4042 & 1.2813 \\
8  & 1.4703 & 1.3846 & 1.2729 \\
9  & 1.4411 & 1.3662 & 1.2641 \\
10 & 1.4160 & 1.3495 & 1.2562 \\ \hline
\end{tabular}
\end{table}

\begin{table}[ht]
\centering
\caption{Values for the constants $C(r)$ provided by Trevi\~no.}
\label{trevinoconstants2r}
\begin{tabular}{rrrr}
\hline
\rowcolor[HTML]{EFEFEF} 
$r$  & $p_0 = 10^{10}$ & $p_0 = 10^{15}$ & $p_0 = 10^{20}$ \\ \hline
2  & 3.6529          & 3.5851          & 3.5751          \\
3  & 2.5888          & 2.5144          & 2.4945          \\
4  & 2.1914          & 2.1258          & 2.1078          \\
5  & 1.9841          & 1.9231          & 1.9043          \\
6  & 1.8508          & 1.7959          & 1.7757          \\
7  & 1.7586          & 1.7066          & 1.6854          \\
8  & 1.6869          & 1.6384          & 1.6187          \\
9  & 1.6283          & 1.5857          & 1.5654          \\
10 & 1.5794          & 1.5410           & 1.5216          \\ \hline
\end{tabular}
\end{table}

The main aim of this paper is to obtain as many improvements to the size of the constants in Corollary \ref{targetresult} as possible. That is, we wish to prove the following.

\begin{theorem}\label{target}
For $r =2$, Theorem \ref{1over2r} holds with the constants provided in Table \ref{newconstants}, and the condition that $1 \leq N < 2p^\frac{5}{8}$.
For  $3 \leq r \leq 6$ and $p \geq 10^8$, or $7 \leq r \leq 10$ and $p \geq 10^9$ holds with the constants provided in Table \ref{newconstants} and no restriction on $N$. 
\end{theorem}

\begin{table}[ht]
\centering
\caption{Values for the constants $C(r)$.}
\label{newconstants}
\begin{tabular}{rrrrrrr}
\hline
\rowcolor[HTML]{EFEFEF} 
$r$ & $p_0 = 10^{5}$  & $p_0 = 10^{6}$  & $p_0 = 10^{7}$  & $p_0 = 10^{8}$  & $p_0 = 10^{9}$  & $p_0 = 10^{10}$ \\ \hline
2   & 3.7125          & 3.4682          & 3.3067          & 3.1980          & 3.1259          & 3.0679          \\
3   & 2.7979          & 2.6371          & 2.5131          & 2.4318          & 2.3776          & 2.3358          \\
4   & 2.4157          & 2.2980          & 2.2022          & 2.1513          & 2.0994          & 2.0613          \\
5   & 2.1801          & 2.0981          & 2.0273          & 1.9755          & 1.9419          & 1.9084          \\
6   & 2.0874          & 2.0037          & 1.9424          & 1.8962          & 1.8353          & 1.8054          \\
7   & 1.8948          & 1.8454          & 1.8087          & 1.7820          & 1.7561          & 1.7291          \\
8   & 1.7993          & 1.7609          & 1.7306          & 1.7093          & 1.6894          & 1.6696          \\
9   & 1.7266          & 1.6963          & 1.6692          & 1.6492          & 1.6323          & 1.6186          \\
10  & 1.6720          & 1.6411          & 1.6175          & 1.5991          & 1.5845          & 1.5727          \\ \hline
    &                 &                 &                 &                 &                 &                 \\ \hline
\rowcolor[HTML]{EFEFEF} 
$r$ & $p_0 = 10^{11}$ & $p_0 = 10^{12}$ & $p_0 = 10^{13}$ & $p_0 = 10^{14}$ & $p_0 = 10^{15}$ & $p_0 = 10^{16}$ \\ \hline
2   & 3.0280          & 2.9997          & 2.9790          & 2.9635          & 2.9515          & 2.9421          \\
3   & 2.3025          & 2.2782          & 2.2600          & 2.2461          & 2.2351          & 2.2263          \\
4   & 2.0329          & 2.0117          & 1.9956          & 1.9831          & 1.9733          & 1.9654          \\
5   & 1.8831          & 1.8638          & 1.8487          & 1.8367          & 1.8272          & 1.8194          \\
6   & 1.7825          & 1.7646          & 1.7503          & 1.7388          & 1.7294          & 1.7216          \\
7   & 1.7081          & 1.6914          & 1.6779          & 1.6669          & 1.6577          & 1.6500          \\
8   & 1.6501          & 1.6345          & 1.6219          & 1.6112          & 1.6023          & 1.5946          \\
9   & 1.6029          & 1.5882          & 1.5762          & 1.5661          & 1.5575          & 1.5501          \\
10  & 1.5629          & 1.5499          & 1.5384          & 1.5287          & 1.5205          & 1.5134          \\ \hline
    &                 &                 &                 &                 &                 &                 \\ \hline
\rowcolor[HTML]{EFEFEF} 
$r$ & $p_0 = 10^{17}$ & $p_0 = 10^{18}$ & $p_0 = 10^{19}$ & $p_0 = 10^{20}$ & $p_0 = 10^{50}$ & $p_0 = 10^{75}$ \\ \hline
2   & 2.9345          & 2.9282          & 2.9230          & 2.9185          & 2.8752          & 2.8658          \\
3   & 2.2190          & 2.2128          & 2.2076          & 2.2029          & 2.1503          & 2.1368          \\
4   & 1.9590          & 1.9537          & 1.9493          & 1.9455          & 1.9094          & 1.9011          \\
5   & 1.8130          & 1.8077          & 1.8033          & 1.7996          & 1.7689          & 1.7630          \\
6   & 1.7151          & 1.7097          & 1.7051          & 1.7012          & 1.6715          & 1.6668          \\
7   & 1.6435          & 1.6380          & 1.6333          & 1.6292          & 1.5986          & 1.5947          \\
8   & 1.5883          & 1.5828          & 1.5779          & 1.5738          & 1.5986          & 1.5382          \\
9   & 1.5439          & 1.5384          & 1.5336          & 1.5294          & 1.4959          & 1.4925          \\
10  & 1.5072          & 1.5019          & 1.4972          & 1.4930          & 1.4581          & 1.4548          \\ \hline
\end{tabular}
\end{table}

In this and many other regards, the author is indebted to Trevi\~no, since the method of proof will be essentially the same as his. There are two primary ways one could modify the arguments of \cite{trevino2015} to obtain better constants. For one, Burgess' bound is automatic when $N$ is large enough, since in such a scenario the P\'olya--Vinogradov inequality is stronger. However, in \cite{trevino2015}, the simple estimate
\[\lvert S_\chi(M,N) \rvert \leq \sqrt{q}\log{q},\]
is used. Here, we will use Theorem \ref{pvlargeq} instead. This has the effect of reducing the range of $N$ for which we need to establish Burgess' bound, which in turn allows us to admit smaller constants in said bound. This alone yields a significant gains over the constants in \cite{trevino2015}.

The second technique we employ involves the following counting lemma. 
\begin{lemma}{\textnormal{\cite[Lemma 2.1]{trevino2015}}}\label{v2trev}
Let $p$ be a prime and $A\in \left[28,\tfrac{N}{12}\right)$ and $N<p$ be positive integers. Then, we have
\[V_2 := \sum_{x \imod{p}} v^2(x) \leq 2AN\left(\frac{AN}{p} + \log{1.85A}\right),\]
where
\[v(x) = \left|\left\{(a,n) \in \N\mbox{ }|\mbox{ }1\leq a \leq A\mbox{, }M < n \leq M+N\mbox{ and }n \equiv ax \imod{p}\right\} \right|.\]
\end{lemma}

In the next section, we relax the restrictions on $A$ to extract some additional terms in this estimate. This will allow us to compute $C(r)$ for smaller values of $p_0$. A bonus feature will be that the value of $A$ has more influence on the the size of the bound. However, in both our case and Trevi\~no's case, we note that this estimate seems to be a little less than twice as large as the actual value of $V_2$.

In several cases, determining improved constants $c(r)$ for the case of $r=2$ in Theorem \ref{1overr} allows us to improve upon the constants $C(r)$ across all $r$. The values we determine for $c(2)$ are provided in Table \ref{2constants}. Using these constants and the improvements mentioned above, we arrive at the constants $C(r)$ in Table \ref{newconstants}. 

Once established, these constants can be applied to various number-theoretic problems. In particular, we establish the following improvement to an application, which is essentially \cite[Theorem 1.10]{trevino2015}.

\begin{theorem}\label{kthpowers}
Let $p$ be a prime number and $k > 1$ be a positive divisor of $p-1$. Let $n_{p,k}$ be the least $k$th power non-residue modulo $p$. Fix $\alpha > \frac{1}{4\sqrt{e}}$. Then there is a computable $Y$ (depending only on $\alpha$) for which $n_{p,k} < p^\alpha$ whenever $p \geq p^Y$.
\end{theorem}
In particular, Table \ref{kpossible} lists several such pairs of $\alpha$ and $Y$. Note that \cite{trevino2015} established the pair $(\sfrac{1}{6}, 4732)$.

\begin{table}[ht]
\centering
\caption{Pairs $\alpha$, $Y$ obtained in Theorem \ref{kthpowers}.}
\label{kpossible}
\begin{tabular}{cc}
\hline
\rowcolor[HTML]{EFEFEF} 
$\alpha$  & $Y$ \\ \hline
$\sfrac{1}{4}$ & 83 \\
$\sfrac{1}{5}$ & 334 \\
$\sfrac{1}{6}$ & 3872 \\ \hline
\end{tabular}
\end{table}

We also consider an application of Lezowski and McGown (\cite{lezowski2017}), which uses Burgess' bound to bound the conductors of norm-Euclidean cyclic number fields with prime degree $l$, $3 \leq l \leq 100$. Using the improved constants we obtain, we may establish the following modest improvement to \cite[Proposition 2.4]{lezowski2017}. 

\begin{proposition}\label{conductorbounds}
Table \ref{NEtable} provides unconditional bounds on the conductor $f$ of norm-Euclidean cyclic number fields of odd prime degree $3< l < 100$.
\end{proposition}

\begin{table}[ht]
\centering
\caption{Unconditional bounds on the conductor $f$ of norm-Euclidean cyclic number fields of odd prime degree $3 \leq l < 100$. These bounds improve upon \cite[Proposition 2.4]{lezowski2017} by no more than a factor of 100.}
\label{NEtable}
\begin{tabular}{rccccc}
$l =$    & \cellcolor[HTML]{EFEFEF}3  & \cellcolor[HTML]{EFEFEF}5  & \cellcolor[HTML]{EFEFEF}7  & \cellcolor[HTML]{EFEFEF}11 & \cellcolor[HTML]{EFEFEF}13 \\ \cline{2-6} 
$f < $ & $2.0\cdot 10^{49}$         & $5.1\cdot 10^{53}$         & $7.9\cdot 10^{57}$         & $7.0\cdot 10^{62}$         & $2.7\cdot 10^{64}$         \\ \cline{2-6} 
       &                            &                            &                            &                            &                            \\
$l =$    & \cellcolor[HTML]{EFEFEF}17 & \cellcolor[HTML]{EFEFEF}19 & \cellcolor[HTML]{EFEFEF}23 & \cellcolor[HTML]{EFEFEF}29 & \cellcolor[HTML]{EFEFEF}31 \\ \cline{2-6} 
$f <$  & $8.5\cdot 10^{66}$         & $8.9\cdot 10^{67}$         & $4.8\cdot 10^{69}$         & $5.7\cdot 10^{71}$         & $2.3\cdot 10^{72}$         \\ \cline{2-6} 
       &                            &                            &                            &                            &                            \\
$l =$    & \cellcolor[HTML]{EFEFEF}37 & \cellcolor[HTML]{EFEFEF}41 & \cellcolor[HTML]{EFEFEF}43 & \cellcolor[HTML]{EFEFEF}47 & \cellcolor[HTML]{EFEFEF}53 \\ \cline{2-6} 
$f<$   & $8.2\cdot 10^{73}$         & $6.6\cdot 10^{74}$         & $1.8\cdot 10^{75}$         & $1.1\cdot 10^{76}$         & $1.2\cdot 10^{77}$         \\ \cline{2-6} 
       &                            &                            &                            &                            &                            \\
$l =$    & \cellcolor[HTML]{EFEFEF}59 & \cellcolor[HTML]{EFEFEF}61 & \cellcolor[HTML]{EFEFEF}67 & \cellcolor[HTML]{EFEFEF}71 & \cellcolor[HTML]{EFEFEF}73 \\ \cline{2-6} 
$f<$   & $9.8\cdot 10^{77}$         & $2.0\cdot 10^{78}$         & $1.3\cdot 10^{79}$         & $4.0\cdot 10^{79}$         & $6.8\cdot 10^{79}$         \\ \cline{2-6} 
       &                            &                            &                            &                            &                            \\
$l =$    & \cellcolor[HTML]{EFEFEF}79 & \cellcolor[HTML]{EFEFEF}83 & \cellcolor[HTML]{EFEFEF}89 & \cellcolor[HTML]{EFEFEF}97 &                            \\ \cline{2-5}
$f<$   & $3.3\cdot 10^{80}$         & $8.7\cdot 10^{80}$         & $3.5\cdot 10^{81}$         & $1.9\cdot 10^{82}$         &                            \\ \cline{2-5}
\end{tabular}
\end{table}

\begin{remark}
For the majority of this paper, it suffices to write $C(r)$ or $c(r)$ to represent the constants in Theorems \ref{1overr} and \ref{1over2r}. However, when it is necessary to highlight the dependence on $p_0$, we may also write $C(r;p_0)$ or $c(r;p_0)$. 
\end{remark}

\section{Tighter Estimates for $V_2$}

In estimating $V_2$, we will make use of the following estimates for summatory functions.

\begin{lemma}\label{sum1onn}
For $x \geq 1$, 

\[\sum_{n \leq x}\frac{1}{n} < \log{x} + \gamma + \frac{1}{2x} - \frac{1}{12x^2} + \frac{1}{64x^4}\]
\end{lemma}
\begin{proof}
The lemma can easily be verified for $x = 1$. Note that the approach here is substantially similar to \cite[Lemma 2.8]{dona2019}. For $x >1$, we know from Euler--Maclaurin summation \cite[Theorem B.5]{montgomery2007} that for any integer $K$
\begin{equation}\label{EMharm}
\sum_{1<n\leq x}\frac{1}{n} - \log{x} = \sum_{i=1}^{K}\left(\frac{B_i}{i} -\frac{B_i(\fp{x})}{ix^i}\right) -\int_1^x \frac{B_K(\fp{t})}{t^{K+1}} \, \textrm{d}t,
\end{equation}
where $B_i$ and $B_i(x)$ are the $i$th Bernoulli number and polynomial, respectively. We may rewrite equation \eqref{EMharm} to take advantage of the fact that we know $\lim_{x\to\infty}\left(\sum_{n\leq x} \frac{1}{n} - \log{x}\right) = \gamma$, the Euler--Mascheroni constant. That is, 
\[
\sum_{n\leq x}\frac{1}{n}  - \log{x} = 1-\sum_{i=1}^{K}\frac{B_i}{i} - \int_1^\infty \! \frac{B_K(\fp{t})}{t^{K+1}} \, \textrm{d}t + R_K(x) = \gamma + R_K(x),  
\]
where 
\[R_K(x) = -\sum_{i=1}^{K}\frac{B_i(\fp{x})}{ix^i} + \int_x^\infty \! \frac{B_K(\fp{t})}{t^{K+1}} \, \textrm{d}t = \bigO\left(\frac{1}{x}\right).\]

Let $K=4$. Then, after some rearranging,

\begin{equation*}
    \begin{aligned}
    R_4(x) &- \frac{1}{2x} + \frac{1}{12x^2} - \frac{1}{120x^4} = \\ &-\frac{\fp{x}}{x}\left(1  - \frac{1}{2x} + \frac{1}{6x^2} + \frac{\fp{x}}{2x} +\frac{\fp{x}^2}{3x^2} - \frac{\fp{x}}{2x^2} + \frac{\fp{x}}{4x^3} + \frac{\fp{x}^3}{4x^3} - \frac{\fp{x}^2}{2x^3}\right) \\ &+\int_x^\infty \! \frac{-\frac{1}{30} + \fp{t}^2 + \fp{t}^4 - 2\fp{t}^3}{t^5} \, \textrm{d}t. 
      \end{aligned}
\end{equation*}
The bracketed expression above is positive for $x>1$, while the integral is $\bigO(x^{-4})$, so with a tight enough estimate on the integral, we can improve our bound on $\sum_{n\leq x} n^{-1}$ with two exact lower order terms ($\tfrac{1}{2x} \text{and} -\tfrac{1}{12x^2}$) and one estimated lower order term ($\frac{1}{120x^4}$, with an adjustment no larger than $\frac{1}{120x^4}$ itself). In particular, we have that 
\[\int_x^\infty \! \frac{-\frac{1}{30} + \fp{t}^2 + \fp{t}^4 - 2\fp{t}^3}{t^5} \, \textrm{d}t < -\frac{1}{120x^4} + \int_x^\infty \! \frac{1}{16t^5} \, \textrm{d}t = - \frac{1}{120x^4} + \frac{1}{64x^4}.\]
Therefore, 
\[R_4(x) < \frac{1}{2x} -\frac{1}{12x^2} + \frac{1}{64x^4}.\]
Hence, we have the proposition.
\end{proof}

\begin{lemma}\label{lndond}
For $x \geq 1$, 
\[\sum_{n \leq x} \frac{\log{n}}{n^2} > -\zeta'(2) - \frac{\log{x}}{x} - \frac{1}{x} - \frac{\log(x)}{x^2}.\]
\end{lemma}

\begin{proof}
The proof is similar to that of Lemma \ref{sum1onn}. Here, we take $K=1$ in \cite[Theorem B.5]{montgomery2007} and observe that 
\[\lim_{x\to\infty}\left(\sum_{1 \leq n \leq x}\frac{\log{n}}{n^2} - \frac{\log{x} + 1}{x}\right) = -\zeta'(2).\]
Thus,
\[\sum_{1 \leq n \leq x}\frac{\log{n}}{n^2} = -\zeta'(2) - \frac{\log{x}}{x} - \frac{1}{x} + R_1(x),\]
where 
\[R_1(x) = \frac{\log{x}}{2x^2} - \frac{\fp{x}\log{x}}{x^2} - \int_x^\infty \! \left(\fp{t} - \frac{1}{2} \right)\left(\frac{1}{t^3} - \frac{2\log{t}}{t^3}\right) \, \textrm{d}t.\]
It can verified that $R_1(x)$ is bounded below by $-\tfrac{\log{x}}{x^2}$, establishing the result.
\end{proof}

Alongside the two preceding estimates, we borrow the following lemmas directly from \cite{trevino2015}. 
\begin{lemma}{\textnormal{\cite[Lemma 2.2]{trevino2015}}}\label{phiestimate}
For $x > 1$ real, we have 
\[\sum_{n\leq x} \frac{\phi(n)}{n} \leq \frac{6}{\pi^2}x + \log{x} + 1.\]
\end{lemma}

\begin{lemma}{\textnormal{\cite[Lemma 2.3]{trevino2015}}}\label{phinn}
For $x \geq 1$, 
\[\sum_{n\leq x} \phi(n)n \leq \frac{2}{\pi^2}x^3 + \frac{1}{2}x^2\log{x} + x^2.\]
\end{lemma}
Note that the leading constants in these results are correct (see \cite{pillai1930}, for example). Any savings that could potentially be made here would come from improving the lower order terms, perhaps as in \cite[Hilfssatz 1]{walfisz1963}. 
\begin{lemma}{\textnormal{\cite[Lemma 2.4]{trevino2015}}}\label{lnAond}
For $x \geq 1$, 
\[\sum_{n\leq x} \log(\frac{x}{n}) \leq x - 1.\]
\end{lemma}

From Lemma \ref{v2trev}, recall that
\[v(x) = \left|\left\{(a,n) \in \N\mbox{ }|\mbox{ }1\leq a \leq A\mbox{, }M < n \leq M+N\mbox{ and }n \equiv ax \imod{p}\right\} \right|.\] The following estimate essentially comes from \cite[Lemma 2.1]{trevino2015}. However, some care has been taken to minimize the reliance on a lower bound for the parameter $A$. This will allow us to determine Burgess constants for much smaller $p_0$.  

\begin{lemma}\label{v2bound}
Let $p$ be a prime and $N$ be a positive integer. Let $A > 1$ be an integer satisfying $11A < N$. Then, 
\[V_2 \leq 2AN\left(0.83575\frac{AN}{p} + \frac{6}{\pi^2}\log(e^{\gamma + \frac{1}{2A}-\frac{1}{12A^2}+\frac{1}{64A^4}}A) + \frac{3}{2} + \frac{A-1}{2N} - \frac{1}{A} \right).\]
\end{lemma}

\begin{proof}
Note that in \cite{trevino2015}, $V_2$ counts quadruples $(a_1, a_2, n_1, n_2)$ satisfying $1 \leq a_1, a_2 \leq A$ and $M < n_1, n_2 \leq M+N$ where $a_1n_2 \equiv a_2n_1 \imod{p}$. Trevi\~no concludes that we must have \cite[formula (2.16)]{trevino2015}
\begin{equation*}\label{trev2ino}
V_2 \leq AN + \frac{2N^2}{p}S_1 + \frac{2N}{p}S_2 + 2NS_3 + A^2 - A,
\end{equation*}
where
\begin{equation}\label{s1bound}
S_1 = \sum_{1 \leq a_2 \leq A}\sum_{1\leq a_1 < a_2} \frac{a_1 + a_2}{\max(a_1,a_2)}  = \frac{3}{4}A^2- \frac{3}{4}A,
\end{equation}
\[S_2 = \sum_{1 \leq a_2 \leq A}\sum_{1\leq a_1 < a_2} \frac{a_1 + a_2}{\gcd(a_1,a_2)} = \frac{3}{2}\sum_{1 \leq d \leq A}\sum_{2\leq b_2 \leq \frac{A}{d}}\phi(b_2)b_2,\]
and
\[S_3 = \sum_{1 \leq a_2 \leq A}\sum_{1\leq a_1 < a_2} \frac{\gcd(a_1,a_2)}{\max(a_1,a_2)} = \sum_{1 \leq d \leq A}\sum_{2\leq b_2 \leq \frac{A}{d}}\frac{\phi(b_2)}{b_2}.\]

Using Lemma \ref{phinn} on $S_2$, we see that 
\[S_2 \leq \frac{3\zeta(3)}{\pi^2}A^3 + \frac{3\zeta(2)}{4}A^2\log{A} - \frac{3A^2}{4}\sum_{d \leq A}\frac{\log{d}}{d^2} + \frac{3}{2}A^2\zeta(2).\]
Applying Lemma \ref{lndond} to the above sum, 
\begin{equation}\label{s2bound}
S_2 \leq \frac{3\zeta(3)}{\pi^2}A^3 + \frac{3\zeta(2)}{4}A^2\log{A} - \frac{3}{4}A^2\left(-\zeta'(2) - \frac{\log{A}}{A} - \frac{1}{A} - \frac{\log(A)}{A^2}\right) + \frac{3}{2}A^2\zeta(2),
\end{equation}
for $A > 1$. Using Lemma \ref{phiestimate} on $S_3$, we see that
\[S_3 \leq \frac{6}{\pi^2}A \sum_{d\leq A}\frac{1}{d} +\sum_{d\leq A}\log{\frac{A}{d}}.\]
Applying Lemmas \ref{sum1onn} and \ref{lnAond} to the relevant sums yields, for $A > 1$,
\begin{equation}\label{s3bound}
S_3 \leq \frac{6}{\pi^2}A\log(e^{\gamma + \frac{1}{2A}-\frac{1}{12A^2}+\frac{1}{64A^4}}A) + A-1.
\end{equation}
Now, if we combine equation \eqref{trev2ino} with \eqref{s1bound}, \eqref{s2bound}, and \eqref{s3bound}, we determine that 

\begin{equation}\label{uglybound}\begin{aligned}
    V_2 \leq 2AN \Biggl( &\left(\frac{3A^2\zeta(3)}{\pi^2 p} +\frac{\pi^2A\log{A}}{8p} +\frac{3A\zeta'(2)}{4p} + \frac{3\log{A}}{4p} \right) \\&+ \left(\frac{A\pi^2}{4p} - \frac{3N}{4p}\right) +\left(\frac{3AN}{4p} +\frac{3}{4p} +\frac{3\log{A}}{4Ap}\right) \\&+\left( \frac{6}{\pi^2}\log(e^{\gamma + \frac{1}{2A}-\frac{1}{12A^2}+\frac{1}{64A^4}}A) + \frac{3}{2} + \frac{A-1}{2N} - \frac{1}{A} \right) \Biggr).
    \end{aligned}
\end{equation}
With the conditions on $A$ as stated, we can verify that
\[\frac{3A^2\zeta(3)}{\pi^2 p} +\frac{\pi^2A\log{A}}{8p} +\frac{3A\zeta'(2)}{4p} + \frac{3\log{A}}{4p}  < \frac{11A^2}{16p} \leq \frac{AN}{16p}\]
and 
\[\left(\frac{A\pi^2}{4p} - \frac{3N}{4p}\right) < 0\]
because $N \geq 11 A$. We also have $A \geq 2$, so that
\begin{equation*}\begin{aligned}\left(\frac{3AN}{4p} +\frac{3}{4p} +\frac{3\log{A}}{4Ap}\right) &= \frac{3AN}{4p}
\left(1+ \frac{1}{AN}+\frac{\log{A}}{A^2N}\right) \\ &\leq \frac{6AN}{8p}
\left(1+ \frac{1}{11A^2}+\frac{\ln{A}}{11A^3}\right) \leq \frac{3(1.031)AN}{4p}. 
\end{aligned}\end{equation*}

Combining these estimates in \eqref{uglybound} establishes the result.

\end{proof}

If we restrict $A$ and $N$ in terms of $p$, we can obtain a better estimate for $V_2$, which will help us reduce the power on the logarithm in Theorem \ref{1over2r}. 

\begin{lemma}\label{cutv2bound}
Let $p$ a prime and $N$ be a positive integer. Let $A > 2$ be an integer such that $2AN < p$. Then
\[V_2 \leq 2AN\left(\frac{3}{2} + \frac{6}{\pi^2}\log(e^{\gamma+\frac{1}{2A}- \frac{1}{12A^2} + \frac{1}{64A^4}}A) +\frac{A-1}{2N}-\frac{1}{A}\right). \]
\end{lemma}
\begin{proof}
Under the condition $2AN<p$, \cite[Lemma 4.1]{trevino2015} establishes that 
\[V_2 \leq AN + 2NS_3 + A^2 - A.\]
The proof follows by using \eqref{s3bound} and factoring out $2AN$. 
\end{proof}

\section{Main Theorems}

We will begin by reproducing the proof of Theorem \ref{1over2r}, with modifications according to Lemma \ref{cutv2bound}.

\begin{proof}[Proof of Theorem \ref{1over2r}]
Let $\ub{C}(r)$ be a parameter chosen so that it satisfies $\ub{C}(r) < C(r)$. Then, we may use the trivial bound and our assumption on $N$, to establish the result for $N$ outside the ranges
\begin{equation}\label{nbounds2}
\ub{C}(2)^2p^\frac{3}{8}\sqrt{\log{p}} < N < 2p^\frac{5}{8},
\end{equation}
when $r=2$
or, using Burgess for $r-1$,
\[\ub{C}(r)^rp^{\frac{1}{4}+\frac{1}{4r}}\sqrt{\log{p}} < N < \min\left\{2p^{\frac{1}{2} + \frac{1}{2r}},\left(\frac{C(r-1)}{\ub{C}(r)}\right)^{r(r-1)}p^{\frac{1}{4}+\frac{1}{2r}+\frac{1}{4r(r-1)}}\sqrt{\log{p}}\right\}\]
for $r \geq 3$. Now, we may proceed by induction, assuming that for all $h < N$, we have
\[\lvert S_\chi(M,N) \rvert \leq C(r)h^{1-\frac{1}{r}}p^\frac{r+1}{4r^2}(\log{p})^\frac{1}{2r}.\]
Note that we have already established the result for $h \leq \ub{C}(r)^rp^{\frac{1}{4}+\frac{1}{4r}}\sqrt{\log{p}}$. Then, assume that for all $h < N$, we have $\lvert S_\chi(M,N) \rvert \leq  C(r)h^{1-\frac{1}{r}}p^\frac{r+1}{4r^2}(\log{p})^\frac{1}{2r}$. For such an $h$, we can shift our character sum to yield
\[S_\chi(M,N) = \sum_{n=M+1}^{M+N} \chi(n+h) + \sum_{n=M+1}^{M+h}\chi(n) - \sum_{n=M+N+1}^{M+N+h}\chi(n).\]
In anticipation of using our inductive hypothesis on the last two sums in the above equation, we may write
\[S_\chi(M,N) = \sum_{n=M+1}^{M+N} \chi(n+h) + 2\theta(h) E(h), \]
where $\lvert \theta(h) \rvert \leq 1$ and $E(h) = \max_K \lvert S_\chi(K,h) \rvert$. Now, let $A$ and $B$ be real numbers and average over all the shifts of length $h = ab$ where $a \leq A$, $b \leq B$. Doing so establishes
\begin{equation*}
S_\chi(M,N) = \frac{1}{\floor{A}\floor{B}}\sum_{\substack{a\leq A \\ b \leq B}}\sum_{n=M+1}^{M+N}\chi(n+ab) + \frac{2}{\floor{A}\floor{B}}\sum_{\substack{a\leq A \\ b \leq B}}\theta(ab)E(ab).
\end{equation*}

Let
\[V:= \sum_{x\imod{p}} v(x) \left|\sum_{b\leq B} \chi(x+b) \right|.\]
Then we have
\begin{equation}\label{splitting}
    \left| S_\chi(M,N) \right| \leq \frac{V}{\floor{A}\floor{B}} + \frac{2}{\floor{A}\floor{B}}\sum_{\substack{a\leq A \\ b \leq B}}E(ab).
\end{equation}
Define
        \begin{equation*}
            {V_1 := \sum_{x \imod{p}}v(x),} \hspace{10pt} {V_2 := \sum_{x \imod{p}} v^2(x),}\mbox{ and }{W := \sum_{x\imod{p}} \left| \sum_{1 \leq b \leq B} \chi(x+b) \right|^{2r}},
        \end{equation*}
and apply H\"older's inequality to $V$, with conjugates $\tfrac{2r-1}{2r}$ and $2r$ to get,
\[V \leq \left(\sum_{x \imod{p}} v(x)^\frac{2r-2}{2r-1}v(x)^\frac{2}{2r-1}\right)^\frac{2r-1}{2r}W^\frac{1}{2r}.\]
We apply H\"older's inequality a second time to the first sum, with conjugates $\tfrac{2r-1}{2r-2}$ and $2r-1$, resulting in
\begin{equation*}
    V \leq V_1^{1-\frac{1}{r}}V_2^{\frac{1}{2r}}W^{\frac{1}{2r}}.
\end{equation*}
We already have bounds for each of $V_1$, $V_2$, and $W$. Trivially, $V_1 = \floor{A}N \leq AN$. Using Lemma \ref{cutv2bound}, we bound $V_2$ (meaning we insist that $A > 2$ and $2AN < p$), and, using Theorem \ref{trevinoweil}, we bound $W$ for $r \leq 9B$. With these bounds in hand, and letting $k = AB/N$, we have
\begin{equation}\label{piece1}
    \begin{aligned}
       \frac{V}{\floor{A}\floor{B}} &\leq \frac{V_1^{1-\frac{1}{r}}V_2^{\frac{1}{2r}}W^{\frac{1}{2r}}}{\floor{A}\floor{B}} \\  & \leq \frac{A}{A-1}\cdot\frac{B}{B-1}\cdot\frac{N^{1-\frac{1}{r}}}{k^{\frac{1}{2r}}}\cdot\frac{(2WB)^\frac{1}{2r}}{B} \\ &\cdot\left(\frac{6}{\pi^2}\log(e^{\gamma + \frac{1}{2A}-\frac{1}{12A^2}+\frac{1}{64A^4} + \log(\frac{k\nu_2(p)}{B}}) + \frac{3}{2} + \frac{k}{2B} - \frac{1}{2\nu_r(p)} - \frac{1}{A}\right)^\frac{1}{2r}.
    \end{aligned}
\end{equation}
Inside the brackets, we have replaced $N$ with its upper bound,
\[\nu_2(p) =2p^{\frac{5}{8}}\]
or, for $r \geq 3$,
\begin{equation*}
    \nu_r(p) =\min\left\{2p^{\frac{1}{2} + \frac{1}{2r}},\left(\frac{C(r-1)}{\ub{C}(r)}\right)^{r(r-1)}p^{\frac{1}{4}+\frac{1}{2r}+\frac{1}{4r(r-1)}}\sqrt{\log{p}}\right\}.
\end{equation*}

We wish to minimize the right-hand side of \eqref{piece1}. We can start by fixing $B$ so that  $\tfrac{(2WB)^\frac{1}{2r}}{B}$ is minimized. One sees that such a $B$ is 
\begin{equation}\label{bfix}
((2r-3)!!(r-1))^\frac{1}{r}p^\frac{1}{2r}.
\end{equation}
Making the choice \eqref{bfix} for $B$, we determine that 
\begin{equation}\label{miniest}
\frac{(2WB)^\frac{1}{2r}}{B} = 2^\frac{1}{2r}\left((2r-3)!!(r-1)\right)^\frac{1}{2r^2}\left(\frac{(2r-3)(2r-1)(r-1)+1}{(2r-3)(r-1)}\right)^\frac{1}{2r}  p^\frac{r+1}{4r^2}.
\end{equation}
One may note that this exact expression is an improvement upon \cite[formula (3.9)]{trevino2015}. For example, with $r = 2$, we have $8^\frac{1}{4}p^\frac{3}{16}$ in place of $12^\frac{1}{4}p^\frac{3}{16}$. 

For the error term, recall that the induction hypothesis implies
\[E(ab) \leq C(r)(ab)^{1-\frac{1}{r}}p^\frac{r+1}{4r^2}\left(\log{p}\right)^\frac{1}{2r}\]
and thus,
\begin{equation}\label{piece2}
    \begin{aligned}
        &\frac{1}{p^\frac{r+1}{4r^2}\left(\log{p}\right)^\frac{1}{2r}}\frac{2}{\floor{A}\floor{B}}\sum_{\substack{a \leq A \\ b \leq B}}\theta(ab)E(ab)  \leq \frac{2C(r)}{AB}\sum_{\substack{a \leq A \\ b \leq B}} (ab)^{1-\frac{1}{r}} \\
        &\leq \frac{2C(r)}{AB}\left(\int_1^{A+1} \! t^{1-\frac{1}{r}} \, \textrm{d}t \right)\left(\int_1^{B+1} \! t^{1-\frac{1}{r}} \, \textrm{d}t \right)\frac{AB}{(A-1)(B-1)} \\
        &\leq C(r)(AB)^{1-\frac{1}{r}}\frac{2}{\left(2-\frac{1}{r}\right)^2}\left(\frac{(A+1)(B+1)}{AB}\right)^{2-\frac{1}{r}}\frac{AB}{(A-1)(B-1)} \\ 
        &= \frac{2r^2}{(2r-1)^2}C(r)k^{1-\frac{1}{r}}N^{1-\frac{1}{r}}\left(\frac{(A+1)(B+1)}{AB}\right)^{2-\frac{1}{r}}\frac{AB}{(A-1)(B-1)}.
    \end{aligned}
\end{equation}

Combining \eqref{piece1}, \eqref{miniest}, and \eqref{piece2} in \eqref{splitting}, we determine that

\begin{equation}\label{mainbound}
    \begin{aligned}
        \frac{\left| S_\chi(M,N) \right|}{N^{1-\frac{1}{r}}p^\frac{r+1}{4r^2}(\log{p})^\frac{1}{2r}} &\leq \frac{2^\frac{1}{2r}AB}{(A-1)(B-1)}\left((2r-3)!!(r-1)\right)^\frac{1}{2r^2}\left(\frac{(2r-3)(2r-1)(r-1)+1}{(2r-3)(r-1)}\right)^\frac{1}{2r} \\ &\cdot\left(\frac{6}{\pi^2}\log(e^{\gamma + \frac{1}{2A}-\frac{1}{12A^2}+\frac{1}{64A^4} + \log(\frac{k\nu_2(p)}{B}}) + \frac{3}{2} + \frac{k}{2B} - \frac{1}{2\nu_r(p)} - \frac{1}{A}\right)^\frac{1}{2r} \\ &+\frac{2r^2}{(2r-1)^2}C(r)k^{1-\frac{1}{r}}\left(\frac{(A+1)(B+1)}{AB}\right)^{2-\frac{1}{r}}\frac{AB}{(A-1)(B-1)}.
    \end{aligned}
\end{equation}
If we set the right hand side of \eqref{mainbound} equal to $C(r)$ and solve, we find that
\begin{equation*}
\begin{aligned}
    C(r) = &\frac{2^\frac{1}{2r}AB}{(A-1)(B-1)}\left((2r-3)!!(r-1)\right)^\frac{1}{2r^2}\left(\frac{(2r-3)(2r-1)(r-1)+1}{(2r-3)(r-1)}\right)^\frac{1}{2r}  \\ &\cdot\frac{\left(\frac{6}{\pi^2}\log(e^{\gamma + \frac{1}{2A}-\frac{1}{12A^2}+\frac{1}{64A^4} + \log(\frac{k\nu_2(p)}{B}}) + \frac{3}{2} + \frac{k}{2B} - \frac{1}{2\nu_r(p)} - \frac{1}{A}\right)^\frac{1}{2r}}{1-\frac{2r^2}{(2r-1)^2}k^{1-\frac{1}{r}}\left(\frac{(A+1)(B+1)}{AB}\right)^{2-\frac{1}{r}}\frac{AB}{(A-1)(B-1)}}.
\end{aligned}
\end{equation*}
Up to the issue of minimizing $C(r)$, this proves the result. 

Say we have chosen an $r$ and fixed a lower bound $p_0$ for $p$. Note that this fixes $B$ in \eqref{bfix}. To have used Lemma \ref{cutv2bound} we must have had $2 < A < \tfrac{p}{2N}$ and $2AN < p$, and we know that $A= \tfrac{kN}{B}$. Initially, one may take a poor guess for $\ub{C}(r)$, but better guesses yield better constants, so one should iterate this process to determine optimal choices for $\ub{C}(r)$. Having chosen $\ub{C}(r)$, and using \eqref{nbounds2}, we can pick $k$ such that
\begin{equation}\label{krange}
    \frac{2B}{\ub{C}(r)^2p^{\frac{1}{4}+\frac{1}{4r}}\sqrt{\log{p}}} < k < \frac{Bp}{2\nu_r(p)^2},
\end{equation}
which is contained in
\[\frac{2B}{N} < k < \frac{Bp}{2N^2}.\]
The value of $C(r)$ decreases in the parameter $A$. Noting that each choice of $k$ fixes a lower bound on $A$, say $A_0$, we can vary $k$ over \eqref{krange}, evaluating $C(r)$ at $A = A_0$. Taking the value of $k$ which produces the smallest value for $C(r)$, we determine the constants given in Table \ref{newconstants}. 

\end{proof}

The proof of Theorem \ref{1over2r} establishes explicit Burgess constants for a limited range of $N$. If we have access to a version of Theorem \ref{1overr} with $r=2$ and $c(2)$ small enough, we can extend this range for $r > 2$. Here, we prove such a result.

\begin{proof}[Proof of Theorem \ref{1overr} (for $r = 2$)]
The argument will effectively be the same as in the proof of Theorem \ref{1over2r}. Again the proof is by induction, where, in light of Theorem \ref{pvlargeq} and the trivial bound on character sums, we only need to consider $N$ in the range 
\begin{equation*}
\ub{c}(2)^2p^\frac{3}{8}{\log{p}} < N < \frac{(\alpha_1\sqrt{p}\log{p} + \alpha_2\sqrt{p} )^2}{\ub{c}(2)^2p^\frac{3}{8}{\log{p}}},
\end{equation*}
where $\ub{c}(2) < c(2)$. 
Now, our inductive step is to assume that for all $h< N$, we have $\lvert S_\chi(M,N) \rvert \leq c(2)h^\frac{1}{2}p^\frac{3}{16}(\log{p})^\frac{1}{2}$. Therefore, the error term will be 

\begin{equation}\label{pie2}
    \begin{aligned}
        &\frac{1}{p^\frac{r+1}{4r^2}\left(\log{p}\right)^\frac{1}{r}}\frac{2}{\floor{A}\floor{B}}\sum_{\substack{a \leq A \\ b \leq B}}\theta(ab)E(ab) \\  &\leq \frac{2r^2}{(2r-1)^2}c(r)k^{1-\frac{1}{r}}N^{1-\frac{1}{r}}\left(\frac{(A+1)(B+1)}{AB}\right)^{2-\frac{1}{r}}\frac{AB}{(A-1)(B-1)}.
    \end{aligned}
\end{equation}

In light of Lemma \ref{v2bound} (thereby insisting that $2 < A < \tfrac{N}{11}$), we may bound the main term as 
\begin{equation}\label{pie1}
    \begin{aligned}
       \frac{V}{\floor{A}\floor{B}} &\leq  \frac{A}{A-1}\cdot\frac{B}{B-1}\cdot\frac{N^{1-\frac{1}{r}}}{k^{\frac{1}{4}}}\cdot\frac{(2WB)^\frac{1}{4}}{B} \\ &\cdot\left(0.83575\frac{k\nu_2(p)^2}{pB} + \frac{6}{\pi^2}\log(e^{\gamma + \frac{1}{2A}-\frac{1}{12A^2}+\frac{1}{64A^4}}\frac{k\nu_2(p)}{B}) + \frac{3}{2} + \frac{k}{2B} -\frac{1}{2\nu_2(p)} - \frac{1}{A}\right)^\frac{1}{4}.
    \end{aligned}
\end{equation}
Combining \eqref{pie1} and \eqref{pie2} in \eqref{splitting} implies, for $r=2$, 
\begin{equation*}
    \begin{aligned}
        \frac{\left| S_\chi(M,N) \right|}{N^\frac{1}{2}p^\frac{3}{16}(\log{p})^\frac{1}{2}} &\leq \frac{AB}{(A-1)(B-1)}(8)^\frac{1}{4} \\ &\cdot\left(\frac{0.83575\frac{k\nu_2(p)^2}{pB} + \frac{6}{\pi^2}\log(e^{\gamma + \frac{1}{2A}-\frac{1}{12A^2}+\frac{1}{64A^4}}\frac{k\nu_2(p)}{B}) + \frac{3}{2} + \frac{k}{2B} -\frac{1}{2\nu_2(p)} - \frac{1}{A}}{k\log^2{p}}\right)^\frac{1}{4} \\ &+\frac{AB}{(A-1)(B-1)}\frac{8}{9}c(2)k^\frac{1}{2}\left(\frac{(A+1)(B+1)}{AB}\right)^\frac{3}{2}. 
    \end{aligned}
\end{equation*}
If we set the right hand side of \eqref{mainbound} equal to $c(2)$ and solve, we find that

\begin{equation*}
    c(2) = \frac{AB}{(A-1)(B-1)}(8)^\frac{1}{4} \frac{\left(\frac{0.83575\frac{k\nu_2(p)^2}{pB} + \frac{6}{\pi^2}\log(e^{\gamma + \frac{1}{2A}-\frac{1}{12A^2}+\frac{1}{64A^4}}\frac{k\nu_2(p)}{B}) + \frac{3}{2} + \frac{k}{2B} -\frac{1}{2\nu_2(p)} - \frac{1}{A}}{k\log^2{p}}\right)^\frac{1}{4}}{1-\frac{AB}{(A-1)(B-1)}\frac{8}{9}k^\frac{1}{2}\left(\frac{(A+1)(B+1)}{AB}\right)^\frac{3}{2}}. 
\end{equation*}

We may minimize $c(2)$ as we did with $C(r)$, noting that the conditions on $A$ in Lemma \ref{v2bound} require that we have $2 < A < \tfrac{N}{11}$, or rather that we vary $k$ so that
\[ \frac{2B}{c(2)^2p^\frac{3}{8}{\log{p}}} < k < \frac{B}{11}.\]
In light of \eqref{bfix}, one notes that $B = p^\frac{1}{4}$ in this setting. Choosing the $k$ which optimizes $c(2)$ gives us the constants in Table \ref{2constants}.

\begin{table}[ht]
\centering
\caption{Constants $c(2)$.}
\label{2constants}
\begin{tabular}{ccccccc}
\cline{2-7}
            & \cellcolor[HTML]{EFEFEF}$p_0 = 10^{5}$  & \cellcolor[HTML]{EFEFEF}$p_0 = 10^{6}$  & \cellcolor[HTML]{EFEFEF}$p_0 = 10^{7}$  & \cellcolor[HTML]{EFEFEF}$p_0 = 10^{8}$  & \cellcolor[HTML]{EFEFEF}$p_0 = 10^{9}$  & \cellcolor[HTML]{EFEFEF}$p_0 = 10^{10}$ \\ \cline{2-7} 
$\chi$ even & 1.9256                                  & 1.7309                                  & 1.5962                                  & 1.4989                                  & 1.4276                                  & 1.3732                                  \\
$\chi$ odd  & 1.8779                                  & 1.6918                                  & 1.5734                                  & 1.4786                                  & 1.4092                                  & 1.3563                                  \\ \cline{2-7} 
            &                                         &                                         &                                         &                                         &                                         &                                         \\ \cline{2-7} 
            & \cellcolor[HTML]{EFEFEF}$p_0 = 10^{11}$ & \cellcolor[HTML]{EFEFEF}$p_0 = 10^{12}$ & \cellcolor[HTML]{EFEFEF}$p_0 = 10^{13}$ & \cellcolor[HTML]{EFEFEF}$p_0 = 10^{14}$ & \cellcolor[HTML]{EFEFEF}$p_0 = 10^{15}$ & \cellcolor[HTML]{EFEFEF}$p_0 = 10^{16}$ \\ \cline{2-7} 
$\chi$ even & 1.3732                                  & 1.3299                                  & 1.2943                                  & 1.2641                                  & 1.2381                                  & 1.2151                                  \\
$\chi$ odd  & 1.3563                                  & 1.3141                                  & 1.2795                                  & 1.2501                                  & 1.2246                                  & 1.2021                                  \\ \cline{2-7} 
            &                                         &                                         &                                         &                                         &                                         &                                         \\ \cline{2-7} 
            & \cellcolor[HTML]{EFEFEF}$p_0 = 10^{17}$ & \cellcolor[HTML]{EFEFEF}$p_0 = 10^{18}$ & \cellcolor[HTML]{EFEFEF}$p_0 = 10^{19}$ & \cellcolor[HTML]{EFEFEF}$p_0 = 10^{20}$ & \cellcolor[HTML]{EFEFEF}$p_0 = 10^{50}$ & \cellcolor[HTML]{EFEFEF}$p_0 = 10^{75}$ \\ \cline{2-7} 
$\chi$ even & 1.1945                                  & 1.1759                                  & 1.1589                                  & 1.1433                                  & 1.1288                                  & 0.9178                                  \\
$\chi$ odd  & 1.1819                                  & 1.1635                                  & 1.1467                                  & 1.1312                                  & 1.1167                                  & 0.8961                                  \\ \cline{2-7} 
\end{tabular}
\end{table}

\begin{remark}
One can make some additional improvements to $C(r)$ using $c(2)$. Observe that, if 
\[C(2;p_1) \leq c(2; p_0)(\log{p_1})^\frac{1}{4} \leq C(2;p_0),\]
then, in light of Theorem \ref{1overr}, we may replace $C(2;p_0)$ with $c(2)(\log{p_1})^\frac{1}{4}$ for any $p \in \left[p_0,\infty\right)$. Checking this for $r=2$ and $p_1 = 10p_0$ (using the $c(2)$ corresponding to even $\chi$, since they are larger in all cases), we may adjust the constants in Theorem \ref{1over2r} to those in Table \ref{newconstants} for $p_0 \leq 10^9$. In order to minimize the upper bound in \eqref{nbounds2} in the proof of Theorem \ref{1over2r}, we should use these adjusted constants when stepping from $r = 2$ to $r=3$. Using the adjusted constants for $r=2$, we determine the constants for $r =3$, as provided below. One may wish to verify that establishing better constants in Theorem \ref{1overr} for $r=3$ does not admit the same adjustments. 
\end{remark}

\end{proof}

As in \cite[Corollary 1.8]{trevino2015}, we can omit the condition on $N$ in Theorem \ref{1over2r} by using the Burgess bound in Theorem \ref{1overr} with $r=2$. The advantage of having a smaller constant is that our results are now valid for primes as small as $10^8$ (previously we could only take primes as small as $10^{10}$).

\begin{proof}[Proof of Theorem \ref{target}]
We need to establish that if $N \geq 2p^{\frac{1}{2}+\frac{1}{4r}}$, then Theorem \ref{1overr} implies the inequality in Theorem \ref{1over2r}. For $p \geq p_0$, Corollary \ref{1overr} implies

\[\left| S_\chi(M,N) \right| < c(2;p_0)N^{\frac{1}{2}}p^\frac{3}{16}\sqrt{\log{p}}.\]
For $r \geq 3$, this inequality implies 
\[\left| S_\chi(M,N) \right| < C(r;p_0)N^{1-\frac{1}{r}}p^\frac{r+1}{4r^2}\left(\log{p}\right)^\frac{1}{2r}\]
whenever
\begin{equation*}
    N > \left(\frac{c(2;p_0)}{C(r;p_0)}\right)^\frac{2r}{r-2}p^\frac{3r+2}{8r} \left(\log{p}\right)^\frac{r-1}{r-2}.
\end{equation*}
Now, if
\[N \leq \left(\frac{c(2;p_0)}{C(r;p_0)}\right)^\frac{2r}{r-2}p^\frac{3r+2}{8r}\]
then 
\[N < 2p^{\frac{1}{2}+\frac{1}{4r}}\]
whenever 
\begin{equation}\label{factcheck}
\left(\frac{c(2;p_0)}{C(r;p_0)}\right)^\frac{2r}{r-2} < \frac{2p^\frac{1}{8}}{\left(\log{p}\right)^\frac{r-1}{r-2}}.
\end{equation}
Taking all combinations of $p_0$ and $r$ available in Tables \ref{newconstants} and \ref{2constants}, one can verify that \eqref{factcheck} holds for $p > p_0 \geq 10^8$ when $3 \leq r \leq 6$ or $p > p_0 \geq 10^9$ when $7 \leq r \leq 10$.  
\end{proof}

\begin{remark}
One could improve the ranges on $p_0$ in Theorem \ref{target} by making the constant $2$ in the condition on $N$ worse. Since this would be a less restrictive condition, it would result in larger $C(r)$. However, it appears the adjustment that would be necessary to extend Theorem \ref{target} to $p_0 \geq 10^5$ would cause the constants to be much larger than desirable. For this reason, we make no adjustment and accept $p_0 \geq 10^9$ in Theorem \ref{target}.  
\end{remark}

\section{Least $k$th Power Non-Residues}

We wish to improve upon the result of Trevi\~no \cite[Theorem 1.10]{trevino2015}, which established that there is a $k$th power non-residue $\imod{p}$ smaller than $p^\frac{1}{6}$ for all primes greater than $10^{4732}$. The approach used by Trevi\~no could be used to establish a result for $p^\alpha$ up to the bound $\alpha > \tfrac{1}{4\sqrt{e}}$. We apply our constants in Trevi\~no's proof to establish a similar result for $p^\alpha$ where $\alpha = \tfrac{1}{4}, \tfrac{1}{5}, \tfrac{1}{6}$. The proof starts with the following lower bound for character sums. 

\begin{lemma}{\textnormal{\cite[Lemma 5.3]{trevino2015}}}\label{vinotrick}
Let $x \geq 286$ be a real number, and let $y= x^{\frac{1}{\sqrt{e}}+\delta}$ for some $\delta > 0$. Let $\chi$ be a non-principal character $\imod{p}$ for some prime $p$. If $\chi(n) =1$ for all $n \leq y$, then 
\[\left| \sum_{n\leq x} \chi(n) \right| \geq x\left(2\log(\delta \sqrt{e} +1) - \frac{1}{\log^2{x}} - \frac{1}{\log^2{y}} - \frac{1}{x}\right).\]
\end{lemma}

We are now in a position to prove Theorem \ref{kthpowers}.
\begin{proof}[Proof of Theorem \ref{kthpowers}]
Let $k$ be a divisor of $p-1$ and $\chi$ be a non-principal character modulo $p$ of order $k$. Then, if $\chi(n) \neq 1$, $n$ is a $k$th power non-residue modulo $p$. Fix $\alpha$ to be one of $\tfrac{1}{4}$, $\tfrac{1}{5}$, or $\tfrac{1}{6}$. Let $x \geq 286$ and $y = x^{\frac{1}{\sqrt{e}}+\delta} = p^\alpha$, where $\delta > 0$ will be a constant that is determined when $x$ is fixed later in the proof. If we suppose that for all $n \leq y$ we have $\chi(n) = 1$, then by comparing Lemma \ref{vinotrick} with Theorem \ref{target}, we have that

\[C(r)x^{1-\frac{1}{r}}p^\frac{r+1}{4r^2}\left(\log{p}\right)^\frac{1}{2r}  \geq x\left(2\log(\delta \sqrt{e} +1) - \frac{1}{\log^2{x}} - \frac{1}{\log^2{y}} - \frac{1}{x}\right).\]
If we take $x = p^{\frac{1}{4} + \frac{1}{2r}}$, then we have
\begin{equation}\label{contra}
    C(r)p^{\frac{\log{log{p}}}{2r\log{p}}-\frac{1}{4r^2}}\geq 2\log(\delta \sqrt{e} +1) - \frac{1}{\log^2{x}} - \frac{1}{\log^2{y}} - \frac{1}{x}.
\end{equation}

One observes that, as $p$ increases, the left-hand side of \eqref{contra} eventually decreases to 0, while the right-hand side eventually increases to $2\log(\delta \sqrt{e} +1)$. Therefore, we can obtain a contradiction for $p$ large enough. For each choice of $\alpha$, we will need to take $r$ large enough so that $\delta >0$. The cases $\alpha = \tfrac{1}{4}$ and $\alpha = \tfrac{1}{5}$ are easier, since $\delta >0$ for $r \geq 4$ or $r \geq 7$, respectively. Taking \eqref{contra} with $\alpha = \tfrac{1}{4}$ and $r=4$ gives us $\delta = 0.060136\ldots$ and $C(4;10^{75}) = 1.9011$. With these values, we determine that \eqref{contra} fails when $p \geq 10^{83}$. Taking \eqref{contra} with $\alpha = \tfrac{1}{5}$  and $r = 7$ gives us $\delta = 0.015691\ldots$ and $C(7;10^{75}) = 1.5947$. With these values, we determine that \eqref{contra} fails when $p \geq 10^{334}$. 

In the case of $\alpha = \tfrac{1}{6}$, we take a little more care to ensure a good result. For this case, $\delta$ is only positive once $r \geq 21$, so we must compute $C(r)$ for larger $r$ than were given in Theorem \ref{target}. Trevi\~no gives us a rough sense of how large the primes need to be for Theorem \ref{kthpowers} to hold when $\alpha = \tfrac{1}{6}$. That is, he shows $p \geq 10^{4732}$, which suggests that computing $C\left(r;10^{3500}\right)$ should be suitable for our purposes. We have done so for $2 \leq r \leq 25$ and compiled these results in Table \ref{longconstants}.

\begin{table}[ht]
\centering
\caption{Constants $C\left(r;10^{3500}\right)$.}
\label{longconstants}
\begin{tabular}{rrrrrrr}
\cline{2-7}
$r$                         & \cellcolor[HTML]{EFEFEF}2  & \cellcolor[HTML]{EFEFEF}3  & \cellcolor[HTML]{EFEFEF}4  & \cellcolor[HTML]{EFEFEF}5  & \cellcolor[HTML]{EFEFEF}6  & \cellcolor[HTML]{EFEFEF}7  \\ \cline{2-7} 
$C(r)$ & 2.8470                     & 2.1051                     & 1.8821                     & 1.7492                     & 1.6561                     & 1.5859                     \\ \cline{2-7} 
                            &                            &                            &                            &                            &                            &                            \\ \cline{2-7} 
$r$                         & \cellcolor[HTML]{EFEFEF}8  & \cellcolor[HTML]{EFEFEF}9  & \cellcolor[HTML]{EFEFEF}10 & \cellcolor[HTML]{EFEFEF}11 & \cellcolor[HTML]{EFEFEF}12 & \cellcolor[HTML]{EFEFEF}13 \\ \cline{2-7} 
$C(r)$ & 1.5308                     & 1.4862                     & 1.4492                     & 1.4180                     & 1.3913                     & 1.3681                     \\ \cline{2-7} 
                            &                            &                            &                            &                            &                            &                            \\ \cline{2-7} 
$r$                         & \cellcolor[HTML]{EFEFEF}14 & \cellcolor[HTML]{EFEFEF}15 & \cellcolor[HTML]{EFEFEF}16 & \cellcolor[HTML]{EFEFEF}17 & \cellcolor[HTML]{EFEFEF}18 & \cellcolor[HTML]{EFEFEF}19 \\ \cline{2-7} 
$C(r)$ & 1.3478                     & 1.3298                     & 1.3138                     & 1.2995                     & 1.2865                     & 1.2747                     \\ \cline{2-7} 
                            &                            &                            &                            &                            &                            &                            \\ \cline{2-7} 
$r$                         & \cellcolor[HTML]{EFEFEF}20 & \cellcolor[HTML]{EFEFEF}21 & \cellcolor[HTML]{EFEFEF}22 & \cellcolor[HTML]{EFEFEF}23 & \cellcolor[HTML]{EFEFEF}24 & \cellcolor[HTML]{EFEFEF}25 \\ \cline{2-7} 
$C(r)$ & 1.2640                     & 1.2541                     & 1.2450                     & 1.2367                     & 1.2289                     & 1.2217                     \\ \cline{2-7} 
\end{tabular}
\end{table}

For each $r \geq 21$ in Table \ref{longconstants}, we can check inequality \eqref{contra} using $C(r)$ and the appropriate $\delta$ to determine which $r$ gives us the best possible contradiction. It happens that this is when $r=23$, where we have $\delta = 0.006802\ldots$ and \eqref{contra} is false for all $p \geq 10^{3872}$. 
\end{proof}

\section{Norm-Euclidean Cyclic Fields}
In \cite[Proposition 2.4]{lezowski2017}, Burgess' bound is used to provide unconditional upper bounds on the size of the conductor of norm-Euclidean cyclic fields with prime degree $3 \leq l \leq 100$. Here, we update these bounds using the improved Burgess constants.

\begin{proof}[Proof of Theorem \ref{conductorbounds}]Following the proof of \cite[Proposition 5.7]{mcgown2012}, let $100 < q_1 < q_2$ be primes and define $D_2(r)$ by
\[D_2(r) \geq \frac{K_1\left(1 + C(r)^{-1}\right)}{K_2}C(r),\]
where \[K_1 = \left(1+q_1^{\frac{1}{k-1}}\right)\left(1+q_2^{\frac{1}{k-1}}\right)\mbox{ and }K_2 = \left(1-q_1^{-1}\right)\left(1-q_2^{-1}\right).\] In the proof in \cite[pp. 2547-48]{lezowski2017}, we may take $C(r) = C(r;10^{40})$, where $r=4$ for $l=5,7$ and $r=3$ otherwise. Then, by inequality (8.1) in \cite{lezowski2017}, the bound on the conductor for $l > 3$ is given by the smallest $f$ for which
\begin{equation}\label{conductorborder}
f \geq 2.7 D_2(r)^r\left(l-1\right)^r f^\frac{3r+1}{4r}(\log{f})^\frac{5}{2}.
\end{equation}
For $l = 3$, we use $r = 4$ and 
\[f \geq 8 D_2(r)^r 2^r f^\frac{3r+1}{4r}(\log{f})^\frac{5}{2}.\]
We computed $C(3;10^{40}) = 2.1590344\ldots$ and $C(4;10^{40}) = 1.9146092\ldots$, which yields $D_2(3) = 3.5239$ and $D_2(4) = 3.1608$ (rounded up). Then we determined where \eqref{conductorborder} was true to establish the bounds in Table \ref{NEtable}. 

\end{proof}

\section{Acknowledgements} 

Many thanks to Tim Trudgian for bringing this project to my attention, and his supervision throughout. I was honoured to have been able to discuss this project in person with Kevin McGown, Enrique Trevi\~no, and Bryce Kerr. Thanks to them for their time and interest. Thank you to Amir Akbary and Jan-Christoph Schlage-Puchta for their comments and suggestions regarding an earlier version of this paper. Additional thanks to Matteo Bordignon, Ethan Lee, Shehzad Hathi, Thomas Morrill, and Aleksander Simonic for a variety of productive discussions throughout this project. 
\newpage

\bibliography{bibliographic.bib} 
\bibliographystyle{amsplain}

\end{document}